\documentclass[12pt]{article}
\usepackage{amsmath,amssymb,amsbsy,amsfonts,amsthm,latexsym}
\usepackage{amsopn,amstext,amsxtra,euscript,amscd}
\usepackage[all]{xy}

 \makeatletter

\newtheorem{theorem}{Theorem}
\newtheorem{lemma}{Lemma}
\newtheorem{claim}{Claim}
\newtheorem{corollary}{Corollary}
\newtheorem*{proposition*}{Proposition}
\newtheorem{proposition}{Proposition}
\newtheorem{definition}{Definition}

\def\cA{{\mathcal A}}

\def\cB{{\mathcal B}}
\def\cV{{\mathcal V}}
\def\R{{\mathbb R}}

\newcommand{\N}{\mathbb N}
\newcommand{\Z}{\mathbb Z}

\def\ls#1{\langle #1 \rangle}

\def\ignor#1{}
\def\red#1{\mathop{\Longrightarrow}\limits^{#1}}
\def\redr{\red{\rho}}

\def\sred#1{\mathop{\longrightarrow}\limits^{#1}}
\def\sredr{\sred{\rho}}

\begin{document}
\title{$p$-quotients of the G.Higman group}
\author{L. Glebsky}
\maketitle
\begin{abstract}
These notes are based on the mini-course ``On the Graham Higman group'', 
given at the Erwin Schr\"odinger Institute in Vienna, 
January 20, 22, 27 and 29, 2016, as a part of the Measured Group 
Theory program.\footnote{This work was partially supported by 
the European Research Council (ERC) grant no. 259527 of
G. Arzhantseva, part of this work was done in the Nizhny Nivgorod University 
and supported by the RSF (Russia) grant 14-41-00044. The stay in
Vienna was supported by ERC grant no. 259527 of G. Arzhantseva.} 
The main purpose is to describe p-quotients of 
the Higman group $H(k)$ for $p|(k-1)$. 
(One may check that the condition $p|(k-1)$ is necessary for the existence 
of such quotients.)
\end{abstract}
\section{Higman group} 
Consider the Higman group
$H(k)=\ls{a_0,\dots,a_{3}\;|\;\{a^{-1}_ia_{i+1}a_i=a_{i+1}^k,\;i=0,...,3\}}$, 
$i+1$ is taking $\mod 4$ here.
It may be constructed as  successive  amalgamated free products, starting from 
Baumslag-Solitar group 
BS(1,k)=$\langle a_0,a_1\;|\;a^{-1}_0a_{1}a_0=a_{1}^k \rangle$ :
$$
B_3\langle a_0,a_1,a_2\rangle=\langle a_0,a_1\;|\;a^{-1}_0a_{1}a_0=a_{1}^k \rangle
\mathop{*}\limits_{a_1\leftrightarrow a_1}\langle a_1,a_2\;|\;a^{-1}_1a_{2}a_1=a_{2}^k \rangle.
$$
Similarly, one may construct $B_3\langle a_2,a_3,a_0\rangle$ and 
$$
H_k=B_3\langle a_0,a_1,a_2\rangle\mathop{*}_{a_0\leftrightarrow a_0,a_2\leftrightarrow a_2}  
B_3\langle a_2,a_3,a_0\rangle
$$
The group $H(2)$ was introduced by Graham Higman in \cite{Higman}
as an example of group without finite quotients. Still $H(2)$ has a lot
of quotients, moreover, it is SQ-universal,  \cite{Lindon}. Actually the
proof of \cite{Lindon} works for $H(k)$, $k\geq 2$, so, $H(k)$ is SQ-universal
for any $k\geq 2$. Some other techniques that were used for $H(2)$ seem
to be applicable for $H(k)$, see 
\cite{Hoff,Mart}. But $H(k)$ for $k>2$ have another, compared with $H(2)$, 
behavior with respect to finite quotients. Particularly, $H(k)$ has 
an arbitrary large $p$-quotient for $p|(k-1)$.Moreover, the intersection of the kernels of these quotient maps
intersects trivially with the Baumslag-Solitar subgroups $B(1,k)=\langle a_i,a_{i+1}\rangle<H(k)$.  
Using \cite{Hoff} it implies
the following statement:
\begin{proposition*}
Let $p|(k-1)$ be a prime. Then for any $\varepsilon>0$ there is an 
$n\in\Z$ and a bijection $f : \Z/p^n\Z \to \Z/p^n\Z$
such that $f(x + 1) = kf(x)$
for at least $(1 - \varepsilon)p^n$ elements $x$ of $\Z/p^n\Z$ and
$f (f (f (f (x)))) = x$
for all $x\in\Z/p^n\Z$.
\end{proposition*}
The interesting property of this $f$ is that it behaves ``almost like'' 
a modular exponent ($x\to ak^x\mod p^n$), but all it's cycles are of the length
$4$. Precisely, $f(x)=a(x)k^x\mod p^n$ where $a(x)=a(x+1)$ for almost all
$x$ ($a(x)$ is ``almost a constant''). Maybe, the existence of such functions
explains the difficulty in proving estimates for the number of small cycles in
repeated modular exponentiation, \cite{GS}. The case $k=2$ is not follows from this notes. 
So, the existence of $f$ for
$k=2$ is an open question.  

Let $X$ be a group (or other algebraic system), $x\in X$ and $\phi:X\to Y$ be 
a homomorphism. We systematically, abusing notations, will write $x$ to denote 
$\phi(x)$
if it is clear from the context that we are dealing with an element of $Y$. 
If there are $\phi_j:X\to Y_j$ we may say ``$x$ of $Y_j$'' to denote $\phi_j(x)$.    
      
\section{$p$-quotients of a group and it's $p$-central series.}\label{sec_p-quotient}

Let $G$ be a group. For $S\subseteq G$ let $\langle S\rangle$ denote the subgroup
of $G$ generated by $S$.
For $g,h\in G$ let $[g,h]=g^{-1}h^{-1}gh$ denote the commutator of $g$ and $h$.
For $H_1,H_2<G$ let $[H_1,H_2]=\langle \{[g,h]\;|\;g\in H_1,\;h\in H_2\}\rangle <G$ denote 
the commutator subgroup of $H_1$ and $H_2$.
The $p$-central series $G_1,G_2,\dots$ of a group $G$ is defined as
$$
G_1=G,\;\;G_{i+1}=G_i^p[G_i,G].
$$ 
It is clear by the definition that $G/G_{i+1}$ is the maximal $p$-quotient of $G$ of
$p$-class at most $i$. There is another, equivalent, definition of $G_k$. Let
$G_{[i]}$ be the lower central series for $G$:
$$
G_{[1]}=G,\;\; G_{[i+1]}=[G_{[i]},G].
$$

{\bf Exercise 1.} Show that $G_n=\langle G_{[i]}^{p^j},\;i+j=n\rangle$. Hint. Using the commutator 
identities (Theorem 5.1 (Witt-Hall identities) of \cite{Mag1}) show that
$$
[u,v^p]=[u,v]^p\cdot [[u,v],v]\cdot [[u,v^2],v]\cdot \dots \cdot [[u,v^{p-1}],v].
$$ 
Particularly, this implies  
that $[G_{[i]},G_{[j]}^{p^k}]\subseteq [G_{[i]},G_{[j]}^{p^{k-1}}]^p[G_{[i+j]},G_{[j]}^{p^{k-1}}]$. 
Show that $[G,G^{p^k}_{[i]}]\subseteq \langle G^{p^j}_{[r]},\; r+j=i+k+1\rangle$. 
Then apply induction on $n$.

\section{Calculating of $p$-quotients.}
Let $\Z_{p^n}=\Z/p^n\Z$. Consider the non-commutative ring
$\Z_{p^n}[\bar x]$ of polynomials with non-commutative (but associative) variables 
$\bar x =(x_0,\dots,x_m)$ over $\Z_{p^n}$. The ring $\Z_{p^n}[\bar x]$ contains finite 
subring
$\Z_{p^n}[p\bar x]$ of polynomials $f(p\bar x)$. It is clear that each monomial
of $f(p\bar x)$ of order $k$ is divisible by $p^k$. Inside of 
$\Z_{p^n}[p\bar x]$ there is a group $\Gamma=\ls{(1+px_0),\dots,(1+px_m)}$, generated
by $(1+px_i)$. Notice, that 
$$
(1+px)^{-1}=\sum_{j=0}^{n-1}(-p)^jx^j
$$
\begin{lemma}\label{lm_power}
 $(1+px_i)^{p^n}=1$ in $\Z_{p^n}[\bar x]$. In other words $j\to (1+px_i)^j$ is a function
 $\Z_{p^n}\to\Z_{p^n}[\bar x]$.
\end{lemma}

\begin{lemma}[Jacobson]\label{lm_free}
$\Gamma$ is isomorphic to $F/F_{n}$ where $F$ is a free group of rank $m+1$.
\end{lemma} 
We prove this lemma in Section~\ref{sec_lm_free}.
Let $I$ be a (two-sided) ideal in $\Z_{p^n}[\bar x]$ and 
$\phi:\Z_{p^n}[\bar x]\to \Z_{p^n}[\bar x]/I$ be a natural map. Let
$\Gamma_I=\phi(\Gamma)$. We are going to use $\Gamma_I$ for calculating $G/G_n$
as follows. Let 
$$
G=\ls{a_0,\dots,a_m\;|\; u_i(a_0,\dots,a_m)=w_i(a_0,\dots,a_m),\;i=0,\dots,k}.
$$
Let 
$g_i=p^{-\alpha_i}(u_i(1+px_0,\dots,1+px_k)-w_i(1+px_0,\dots,1+px_k))$, where $p^{\alpha_i}$ 
devides all coefficients of $u_i(1+px_0,\dots,1+px_k)-w_i(1+px_0,\dots,1+px_k)$ and 
being maximal with this property.
Consider 
$I=I(g_i,\;i=0,\dots,k)$, an ideal
in $\Z_{p^n}[\bar x]$ generated by $g_i$. 

\begin{lemma}\label{lm_G}
$\Gamma_I$ is a homomorphic image of $G/G_{n}$.
\end{lemma} 
\begin{proof}
By construction $\Gamma_I$ is a homomorphic image of $G$ under $\phi:a_i\to (1+px_i)$.
It is easy to check that $\phi(G_{n})=\{1\}$.
\end{proof}
When $\Gamma_I$ is isomorphic to $G/G_n$?
Probably the answer is the following: if $p\neq 2$ then $\Gamma_I$ is isomorphic
to $G/G_n$; for $p=2$ there are $G$ such that $\Gamma_I$ is not isomorphic to $G/G_n$.  
 Probably, it is well known. Otherwise, one   
may try to use similar technique as in the proof of Lemma~\ref{lm_free} 
(see, Section~\ref{sec_lm_free}) 
taking into account the solution of the dimension subgroup problem.
See \cite{Gupta} and the bibliography therein for the dimension subgroup problem.
   
\section{$p$-quotients of $H(k)$, $p|(k-1)$.} 
$H(k)=\ls{a_0,\dots,a_{3}\;|\;\{a_{i+1}a_i=a_ia_{i+1}^k,\;i=0,...,3\}}$,
here $i+1$ is taken $\mod 4$. This is a presentation of the Higman group without 
inversion. 
Let $p|(k-1)$ then 
substituting $a_i=(1+px_i)$ leads
\begin{equation}\label{eq_g}
g_i=\frac{1}{p^2}(a_{i+1}a_i-a_ia_{i+1}^k)=
x_{i+1}x_i-x_ix_{i+1}+Q_0(x_{i+1})+pQ_1(x_i,x_{i+1}).
\end{equation}
Our aim is to study $I=I(g_0,g_1,g_2,g_3)$ in $\Z_{p^n}[x_0,\dots,x_3]$, or
precisely, $\Z_{p^n}[x_0,\dots,x_3]/I$ and $\Gamma_I$. 
To this end we introduce some notions. 
\begin{definition}\label{def_algebra}
A (non-commutative) ring $A$ is called to be an algebra over $\Z_{p^n}$ if 
\begin{itemize}
 \item $A$ has the unity $1\in A$.
 \item Multiplication is $\Z_{p^n}$-bilineal.
\end{itemize}
An algebra $A$ is tame if
$A$ is a free $\Z_{p^n}$-modul with a free basis $\cA\ni 1$.
\end{definition}
All algebras over $\Z_{p^n}$ we deal with are tame. So, in what follows we use
just ``algebra'' to denote ``tame algebra''. 
\begin{definition}\label{def_zappa-szep_alg}
Let $A$ (resp. $B$) be a $\Z_{p^n}$-algebra  and  $\cA\ni 1$ (resp. $\cB\ni 1$) 
be a set of it's free generators as a $\Z_{p^n}$-modul. 
A Zappa-Szep product $C=A\bowtie B$ is a $\Z_{p^n}$ algebra such that 
\begin{itemize}
 \item $C$ contains isomorphic copies of $A$ and $B$ such that 
$A\cap B=\Z_{p^n}\cdot 1$ (we have fixed such a  copies of $A$ and $B$  
and denote them by the same letters).
 \item The set $\{ab\;|\;a\in\cA,\;b\in\cB\}$, as well as the set 
$\{ba\;|\;a\in\cA,\;b\in\cB\}$, forms a free basis of $C$ as a $\Z_{p^n}$-modul.
\end{itemize}
\end{definition}
{\bf Exercise 2.} It looks that by definition one should say that $C=A\bowtie B$
with respect to $\cA$ and $\cB$. 
Show that $C=A\bowtie B$ with respect to any free bases  $\cA'\ni 1$ and $\cB'\ni 1$.
\begin{definition}\label{def_zappa_szep_group}
 Let $K,L$ be a groups. A Zappa-Szep product $G=K\bowtie L$ of groups $K$ and $L$ is a 
group such that
 \begin{itemize}
  \item $G$ contains isomorphic copies of $K$ and $L$. We fix such copies and assume that 
   $K,L<G$.
  \item $K\cap L=\{1\}$ in $G$.
  \item $G=KL$. (This easily implies that $G=LK$.) 
 \end{itemize}
\end{definition}
{\bf Remark.} 
Of course, the definitions do not imply that a Zappa-Szep product is uniquely defined by 
a pair of algebras o groups. Really, in order to define $X\bowtie Y$ uniquely 
(up to isomorphism) one needs a  function $com:X\times Y\to Y\times X$ which describe
how the elements  of $X$ commute with elements of $Y$. So, in some sense, $Z=X\bowtie Y$
is an abuse of notation. In any case, when we use $Z=X\bowtie Y$, the structure of $Z$
will be described.   
\begin{theorem}\label{th_algebra}
 Let $I=I(g_0,g_1,g_2,g_3)$ be an ideal in $\Z_{p^n}[\bar x]$ generated by $g_i$ of 
Eq.(\ref{eq_g}). 
 Then $\Z_{p^n}[\bar x]/I= \Z_{p^n}[x_0,x_2]\bowtie\Z_{p_n}[x_1,x_3]$.
\end{theorem}
Recall, that $G_i$ denotes the $i$-th term of the $p$-central series of group $G=G_0$. 
Also we suppose that $p|(k-1)$. 
\begin{corollary}\label{cor_group}
 There is a surjective  homomorphism $H(k)/H_n(k)\to F/F_n\bowtie F/F_n$, 
where $F$ is a free group of rank $2$.
\end{corollary}
\begin{proof}
By Lemma~\ref{lm_G} $H(k)/H_n(k)$ surjects into $\Gamma_I$. 
By Theorem~\ref{th_algebra} and 
Lemma~\ref{lm_free} there are $S,\tilde S<\Gamma_I$ such that
\begin{itemize}
 \item $S=\ls{a_0=(1+px_0),a_2=(1+px_2)}$, $\tilde S=\ls{a_1=(1+px_1),a_3=(1+px_3)}$;
 \item $S$ and $\tilde S$ are isomorphic to $F/F_n$;
 \item $S\cap\tilde S=\{1\}$;
\end{itemize}
It follows that $|S\tilde S|=|S|\cdot|\tilde S|$ and 
$S\tilde S\leq\Gamma_I$. So, we have to prove that 
$\Gamma_I\subset S\tilde S$.
To this end it suffices to show that in $\Gamma_I$ there exist relations removing appearance of  
$a_1^{m}a_0^{r}$, $a_1^{m}a_2^{r}$,  $a_3^{m}a_0^{r}$  and $a_3^{m}a_2^{r}$. 
By Lemma~\ref{lm_power} we may assume that $m,r\in\Z_{p^n}$.
The relation $a_1 a_0=a_0 a_1^k$ implies $a^m_1 a_0^r=a_0^r a_1^{mk^r}$. 
The relation $a_2 a_1=a_1 a_2^k$ implies $a_1^m a_2^r=a_2^{rk^{-m}}a_1^m$.
Considerations for the other indexes are the same.
Notice here, that the condition $p|(k-1)$ implies that $r\to k^r$ is 
a well definite function $\Z_{p^n}\to\Z_{p^n}$.  
\end{proof}
\section{What is going on}

Consider  $w\in H_k$ as a word $a_{i_1}^{n_1}a_{i_2}^{n_2}\dots$ but now
suppose that $n_j$ are in some commutative ring, in our example, 
$n_j\in\Z_{p^n}$. We get a new group $\tilde H_k$ wich is a homomorphic
image of $H_k$.  
 Now, the relation $a_1^k=a_0^{-1}a_1a_0$ implies as well
the relation $a_1^{{1}/{k}}=a_0a_1a_0^{-1}$. 
We also need that $j\to k^j$ is well defined $\mod p^n$.
After that
any element of $\tilde H_k$ may be written as $wu$, where
$w=a_0^{n_1}a_2^{m_1}\dots a_2^{m_j}$, $n_i,m_i\in\Z_{p^n}$ and, 
similarly, $u=u(a_1,a_3)$. So, we have 
$\tilde H_k=(\Z_{p^n}*\Z_{p^n})\bowtie(\Z_{p^n}*\Z_{p^n})$
The whole problem is to show that the natural homomorphism 
$\Z_{p^n}*\Z_{p^n}\to F/F_n$ is compatible with the $\bowtie$ structure. (
Here $F$
is a rank 2 free group). Which, probably, may be done another way as well...
  
Let restate all it more formally.
On can easy to check that if $G_1=H_1\bowtie K$ and $G_2=H_2\bowtie K$ then
$G_1\mathop{*}\limits_{K=K}G_2=(H_1*H_2)\bowtie K$. For $H_k$,
$p|(k-1)$ we have the following:
\begin{itemize}
\item $BS(1,k)\to BS_{p^n}(1,k)=\Z_{p^n}\ltimes_k\Z_{p^n}$,
\item $B_3\to BS_{p^n}\mathop{*}\limits_{Z_{p^n}} BS_{p^n}=(\Z_{p^n}*\Z_{p^n})\bowtie\Z_{p^n}$, 
where we amalgamate the different factors of $BS_{p^n}$. 
\item Finally we get $H_k\to (\Z_{p^n}*\Z_{p^n})\bowtie (\Z_{p^n}*\Z_{p^n})$.
\end{itemize}

For the case $k=2$ we may do similar things. We need a unitary commutative ring 
$R$ such  that function $r\to 2^r$ is defined in $R$. Such a ring exists, for
example the real numbers $\R$. 
So, we may embed $BS(1,2)\hookrightarrow \R\ltimes_2\R$, where
for $(\alpha_i,\beta_i)\in \R\ltimes_2\R$ the multiplication is defined as
$(\alpha_1,\beta_1)(\alpha_2,\beta_2)=(\alpha_1+\alpha_2,2^{\alpha_2}\beta_1+\beta_2)$.
Similarly, there is a homorphism $B_3\to (\R*\R)\bowtie\R$, which is not injective, but
nontrivial. And finally we obtain $H_2\to H_{\R}=(\R*\R)\bowtie(\R*\R)$. 
(Any element of $(\R*\R)\bowtie(\R*\R)$ is of the form $wu$, 
$w=a_{0}^{\alpha_1}a_2^{\beta_1}\dots a_2^{\beta_k}$, $\alpha_i\beta_i\in \R$ and, similarly, 
$u=u(a_1,a_3)$.) Actually, $H_k$ has a nontrivial homomorphic image in $H_{\R}$ for any
$k\in\Z$.

What is the structure of $H_\R$? Changing $a_i\to a_i^\delta$ we may write
$$
H_\R=\langle a_i^\alpha,\; i=0,\dots,3,\;\alpha\in\R\;|\;a_i^{-\alpha}a_{i+1}^{\beta}a_i^{\alpha}=
a_{i+1}^{\beta\exp(\alpha)}\rangle.
$$ 
It is not hard to show that $\langle a_0^\alpha,a_1^\alpha,a_2^\alpha,a_3^\alpha\rangle<H_\R$
is residually finite for all $\alpha\not\in Y$ for a countable set $Y\subset\R$.
\section{Proof Theorem~\ref{th_algebra}}

The algebra $\Z_{p^n}[\bar x]/I$ may be constructed as amalgamated free 
products, similar to the construction of the Higman group itself.
\begin{definition}
Let $A$, $B$ be $\Z_{p^n}$-algebras, $\cA\ni 1$, $\cB\ni 1$ their corresponding 
free generators as $\Z_{p^n}$-modules. A free product $A*B$ is a $\Z_{p^n}$-algebra
that
\begin{itemize}
\item Generated by $1$ and alternating words of letters from 
$\cA'=\cA\setminus \{1\}$ and $\cB'=\cB\setminus\{1\}$ as a free $\Z_{p^n}$-module. 
A word $w=w_1w_2\dots w_r$ is alternating if 
$(w_i,w_{i+1})\in \cA'\times\cB'\cup\cB'\times\cA'$.
\item  It suffices to define product for two alternating words $w=w_1\dots w_r$ and 
$u=u_1\dots u_k$. So, $wu=w_1\dots w_ru_1\dots u_k$ if $(w_r,u_1)$ alternating and
$wu=w_1\dots w_{r-1}(w_r\cdot u_1)u_2\dots u_r$ otherwise. Here $w_r\cdot u_1$ is the 
product in $A$ or $B$.
\end{itemize} 
\end{definition}
For example, $\Z_{p^n}[x]*\Z_{p^n}[x]=\Z_{p^n}[x_0,x_1]$.

{\bf Exercise 3.}
Check that $A*B$ is well defined, independent of $\cA$ and $\cB$,
contains isomorphic copies of $A$ and $B$ with intersection equals to
$\Z_{p^n}\cdot 1$  and satisfies the universal property of 
free product of algebras.

\begin{definition}
Let $C_1=A\bowtie V$ and $C_2=B\bowtie V$. We define $C=C_1\mathop{*}\limits_{V}C_2$ as
$C=(A*B)\bowtie V$. Let $\cA$, $\cB$ and $\cV$ be free bases of $A$, $B$, and $V$,
correspondingly. Any element is a $\Z_{p^n}$-combination of $wv$, where $w$ is 
alternating word of letters from $\cA'$, $\cB'$  and $v\in\cV$. In order
to define multiplication on $C$ it suffices to define $vw$ as a combination of $w_iv_i$.
It could be done using Zappa-Zsep product structure of $C_1$ and $C_2$. Suppose,
for example that $w=a_1b_1\dots a_kb_k$ then
$$
va_1b_1\dots a_kb_k=\sum_{i_1}a^1_{i_1}v_{i_1}b_1\dots a_kb_k=
\sum_{i_1,j_1}a^1_{i_1}b^1_{i_1,j_1}v_{i_1,j_1}\dots a_kb_k=
$$   
$$
\sum_{i_1,j_1,\dots,i_k,j_k}a^1_{i_1}b^1_{i_1,j_1}\dots a^k_{i_1,\dots j_{k-1},i_k}b^k_{i_1,\dots,j_k}
v_{i_1,\dots,j_k} 
$$
\end{definition}
{\bf Exercise 4.}
Check that $C_1\mathop{*}\limits_VC_2$ is well defined, contains isomorphic copies of
$C_1$ and $C_2$ such that $C_1\cap C_2=V$.
Check that $C_1\mathop{*}\limits_VC_2$ satisfies the universal property of amalgamated 
free products of algebras. 

Consider $A_0=\Z_{p^n}[x_0,x_1]/I(g_0)$, where $I(g_0)$ is the ideal generated by $g_0$ of
Eq.(\ref{eq_g}).
\begin{proposition}\label{prop_A0}
$A_0=\Z_{p^n}[x_0]\bowtie\Z_{p^n}[x_1]$
\end{proposition}

We prove the proposition in Section~\ref{sub_propA0}.
Now let as show how Theorem~\ref{th_algebra} follows from Proposition~\ref{prop_A0}.
Let $A_1=\Z[x_1,x_2]/I(g_1)$. Clearly, $A_1$ is isomorphic to $A_0$, so
$A_1=\Z_{p^n}[x_1]\bowtie\Z_{p^n}[x_2]$. Consider 
$A_{01}=A_0\mathop{*}_{\Z_{p^n}[x_1]} A_1$. Notice, that the roles 
of $\Z_{p^n}[x_1]$ in $A_0$ and $A_1$ are different, precisely  
the isomorphism from $A_0$ to $A_1$
maps $\Z_{p^n}[x_1]$ to $\Z_{p^n}[x_2]$ not to $\Z_{p^n}[x_1]$ of $A_1$.
\begin{lemma}\label{lm_A01}
$A_{01}=\Z_{p^n}[x_0,x_2]\bowtie\Z_{p^n}[x_1]$ is isomorphic to 
$\Z_{p^n}[x_0,x_1,x_2]/I(g_0,g_1)$. 
\end{lemma}
\begin{proof}
By construction $g_0=g_1=0$ in $A_{01}$. So, the map $x_i\to x_i$ prolongs
to a surjective  homomorphism $\phi:\Z_{p^n}[x_0,x_1,x_2]/I(g_0,g_1)\to A_{01}$.
Using the universal property of $A_{01}$ define 
$\psi:A_{01}\to \Z_{p^n}[x_0,x_1,x_2]/I(g_0,g_1)$. Notice, that $\psi$ is surjective
as well ($\Z_{p^n}[x_0,x_1,x_2]/I(g_0,g_1)$ is generated by $x_i$). Check that 
$\psi \phi=id:A_{01}\to A_{01}$ by the universal property.
\end{proof}
Similarly, the algebras $A_2=\Z_{p^n}[x_2,x_3]/I(g_2)$,
$A_3=\Z_{p^n}[x_3,x_0]/I(g_3)$, and
$A_{23}=A_2\mathop{*}\limits_{\Z_{p^n}[x_3]}A_3=\Z[x_2,x_0]\bowtie\Z[x_3]$ may be constructed.
Notice, that there exist isomorphism $A_{01}\to A_{23}$ that sends
$x_0\to x_2$, $x_2\to x_0$, $x_1\to x_3$. Now, we may construct
$A=A_{01}\mathop{*}_{\Z_{p^n}[x_0,x_2]} A_{23}$ (we make isomorphism 
$\Z[x_0,x_2]$ of $A_{01}$ with $\Z[x_2,x_0]$ of $A_{23}$ sending $x_0\to x_0$ and 
$x_2\to x_2$). 
Theorem~\ref{th_algebra} follows from the lemma.
\begin{lemma}
$A=\Z_{p^n}[x_0,x_2]\bowtie\Z_{p^n}[x_1,x_3]$ is isomorphic to 
$\Z_{p^n}[x_0,x_1,x_2,x_3]/I(g_0,\dots,g_3)$.
\end{lemma}

\section{Proof of Proposition~\ref{prop_A0}}\label{sub_propA0}

The proof is based on the fact that a polynomial $g_0$ forms a kind of
a Grobner basis (over non-commutative polynomials). We will not define what a Grobner
basis is for non-commutative polynomials. 
Instead, we directly apply a Knuth-Bendix algorithm \cite{Baader} to $\{g_0\}$.

Let $m(y_1,\dots,y_k)$ be a non-commutative monomial or, the same, a word in alphabet
$\{y_1,\dots,y_k\}$, that is, $m(y_1,\dots,y_k)=z_1z_2\dots z_r$, $z_i\in\{y_1,\dots,y_k\}$.
The product of two monomials is just the concatenation. 
A non-commutative polynomial $f$ over $\Z_{p^n}$ is  a ``linear'' combination of monomials
$f=\sum a_im_i$, $a_i\in\Z_{p^n}$.  The product $f_1f_2$ of polynomials is defined 
using the product of monomials by linearity.

Let us return to the study of $A_0=\Z_{p^n}[x_0,x_1]/I(g_0)$.
We call a polynomial $f\in \Z_{p^n}[x_0,x_1]$ left (resp. right) reduced if 
$f=\sum a_{i,j}x_0^ix_1^j$ (resp. $f=\sum a'_{i,j}x_1^ix_0^j$).
Proposition~\ref{prop_A0} is equivalent to the claim.

\begin{claim}\label{cl_1}
For any $f\in \Z_{p^n}[x_0,x_1]$ there exists a left (resp. right) reduced
polynomial $\tilde f$ such that $f-\tilde f\in I$. If $f\in I$ is left (right) reduced then $f=0$. 
\end{claim}
From this point we restrict ourselves to the left case. The right case may be considered 
similarly. 
Define the one step reduction based on equalities $g_0=0$:
$$
x_1x_0\sredr x_0x_{1}-Q_0(x_{1})-pQ_1(x_0,x_{1}) 
$$

Let $f,\tilde f\in \Z_{p^n}[x_0,x_1]$. 
\begin{itemize}
\item  $\tilde f$ is
a one step reduction of $f$  ($f\sredr\tilde f$) if an appearance 
of $x_1x_0$, in some monomial  of $f$ is changed according
to the above described rule; $\tilde f$ is the resulting polynomial (after applying 
associativity and linearity). 
\item $f$ is said to be terminal if no monomial of $f$ contains 
$x_1x_0$. It means that we are unable to apply
$\sredr$ to $f$.
\item We write $f\redr\tilde f$ ($\tilde f$ is a reduction of $f$) if there is 
a sequence $f_0=f,f_1,\dots,f_k=\tilde f$ such that $f_i\sredr f_{i+1}$ 
for $i=0,\dots,k-1$.
\item We write $f\redr\tilde f+$ if $f\redr\tilde f$ and 
$\tilde f$ is terminal.   
\end{itemize}
\begin{proposition}\label{prop_reduction}
\begin{itemize}
\item $f$ is terminal if and only if $f$ is left reduced.
\item There is no infinite sequence $f_0\sredr f_1\sredr f_2\dots$
\item For any non-terminal $f$ there exist a unique $\tilde f$ such 
that $f\redr\tilde f+$.
\end{itemize}
Two last items mean that any sequence $f\sredr f_1\sredr\dots$ of one step reduction 
terminates and the terminal polynomial depends only on $f$.
\end{proposition}
We prove the proposition in Subsection~\ref{sub_prop}. Now let us show how 
Proposition~\ref{prop_reduction} implies Claim~\ref{cl_1}.
We use notation $f=\sredr f'$ (resp. $f=\redr f'$) to denote $f\sredr f'$ or 
$f=f'$ (resp. $f\redr f'$ or $f=f'$).
\begin{lemma}\label{lm_lin}
The map $f=\redr\tilde f+$ is linear, that is, if $f_1=\redr\tilde f_1+$ and
$f_2=\redr\tilde f_2+$ then $a_1f_1+a_2f_2=\redr a_1\tilde f_1+a_2\tilde f_2$, where
$a_1,a_2\in\Z_{p^n}$.  
\end{lemma}
\begin{proof}
Let $a_1f_1+a_2f_2\sredr f'$. It means that we apply reduction to a monomial $m$
of $a_1f_1+a_2f_2$. The monomial $m$ may appear in $f_1$, $f_2$, or 
in the both polynomials. In any case there exist $f'_1$ and $f'_2$ such that
$f_1=\sredr f_1'$, $f_2=\sredr f'_2$ and $f'=a_1f'_1+a_2 f'_2$.  
It follows by induction that if $a_1f_1+a_2f_2\redr f'$ then 
$f'=a_1f'_1+a_2f'_2$ for some $f'_1, f'_2$ such that $f_1=\redr f'_1$ and $f_2=\redr f'_2$.
Now suppose that $a_1f'_1+a_2f'_2$ is terminal but, say, $f'_1$ is not terminal. There
are two possibilities:
\begin{enumerate}
\item $a_1f'_1$ is terminal. In this case, collecting terminal monomials, we may write 
$f'_1=\alpha+\beta$ with $\alpha$ terminal 
and $a_1\beta=0$. In this case further reduction of $f'_1$ does not change $a_1f'_1$.
So, w.l.g. we may assume $f'_1$ to be terminal.
\item  $a_1f'_1$ is not terminal. In this case we may write 
$f'_i=\alpha_i+\beta_i$ with $\alpha_i$ terminal and $a_1\beta_1+a_2\beta_2=0$.
Now, apply the same reduction to $\beta_1$ and $\beta_2$, keeping 
the sum $a_1f'_1+a_2f'_2$ unchanged. 
\end{enumerate}
We are done by Proposition~\ref{prop_reduction}.
\end{proof}
\begin{lemma}\label{lm_ideal}
$f\redr 0+$ if and only if $f\in I(g_0,\dots,g_3)$.
\end{lemma}
\begin{proof}
{\bf Only if.} It is clear by construction that $f\redr\tilde f$ implies that 
$f-\tilde f\in I$.\\
{\bf If.} Let $f\in I$. It means that $f=\sum \alpha_ig_0\beta_i$.
Applying associativity we may write $f=\sum a_im_ig_0m'_i$, where $m_i$ and $m'_i$
are monomials and $a_i\in\Z_{p^n}$. By construction, there is a one step reduction 
$m_ig_0m'_i\sredr 0$. We are done by 
Lemma~\ref{lm_lin} and Propostion~\ref{prop_reduction}. 
\end{proof}
\subsection{Proof of Proposition~\ref{prop_reduction}.}\label{sub_prop}
The first item of Proposition~\ref{prop_reduction} is straightforward.  
So we start with the proof of the second item
of the proposition. 
It is the most difficult part of the proposition and will be used 
for the proof of the third item.
\subsubsection{Proof of the second item.}
 
Let $t=a_jm(x_0,x_1)$ be a term (a monomial with a coefficient).
We are going to measure how an application of one step reduction makes
a term more close to a left reduced polynomial.
To this end we define:
\begin{itemize}
\item $|t|=\min\{k\in\N\;|\;p^kt=0\}$.
\item $n_0(t)$ -- number of $x_0$ in $t$, for example, $n_0(x_1^ix_0^j)=j$.
\item $def(t)$ -- the defect of $t$, the total number of pairs where $x_1$ appears
before $x_0$, for example, $def(x_1^jx_0^kx_1^rx_0^m)=jk+jm+rm$.  
\end{itemize}
To each term we associate the ordered triple $(|t|,n_0(t),def(t))$.
On the set of triple we consider lexicographical order:
$(\alpha,\beta,\gamma)<(\alpha',\beta',\gamma')$ iff 
\begin{itemize}
\item $\alpha<\alpha'$; or 
\item $\alpha=\alpha'$ and $\beta<\beta'$; or
\item $\alpha=\alpha'$, $\beta=\beta'$, and $\gamma<\gamma'$.
\end{itemize}  
Now one may check that 
$(|t|,n_0(t),def(t))>(|t_j|,n_0(t_j),def(t_j))$ if $t\redr\sum_j t_j$. We need the following result.
\begin{lemma}[Dickson] 
Any decreasing (with respect to lexicographic order) sequences in $\N^3$ is finite. 
\end{lemma}  
Consider now the reduction process of $t$ as a tree: 
To each vertex we associate a term in such a way that in any reduction step the 
resulting polynomial is a sum of terms of the leafs of the tree. 
With root we associate $t$. For each reduction of term $t'$ in a leaf $l$ we  
connect the leaf $l$ with  new leafs with all terms appearing in the reduction.
Any descending path in this tree is finite by Dickson Lemma. This tree is $k$-regular
by construction, so the tree is finite and the reduction process terminates.       
\subsubsection{Proof of the third item.}
This uses the Newman's lemma, or Diamond lemma for reduction processes.

Suppose that on a set $X$ a reduction process $\cdot\sred{*}\cdot$ 
(just a relation on $X$) is defined.
Denote by $\red{*}$ it's transitive closure. We say that $x$ is terminal 
if there are no $y\in X$ such that $x\red{*}y$
As before, let $x\red{*}y+$ denotes $x\red{*}y$ and $y$ is terminal.
\begin{lemma}[Diamond lemma] 
Let $\sred{*}$ satisfies the following properties:
\begin{itemize}
\item Any sequence $x_1\sred{*}x_2\sred{*}...$ is finite.
\item $\sred{*}$ is locally confluent, that is, for any $x$, $y_1$ and $y_2$
such that $x\sred{*} y_1$ and $x\sred{*} y_2$ there exists $z\in X$ such that
$y_1\red{*}z$ and $y_2\red{*} z+$.
\end{itemize}
Then $\red{*}$ is globally confluent, that is, for any non-terminal $x$ there
exists unique $y$ such that $x\red{*} y$.
\end{lemma}
So, in order to show Proposition~\ref{prop_A0} it suffices to  
check the second condition of the Diamond Lemma for $\sredr$. 
Let $f\sredr f_1$ and $f\sredr f_2$. If the reduction applies to a different terms
then existence of $f_3$, $f_1\sredr f_3$ and $f_2\sredr f_3$ is trivial.
It suffices to consider $f=ax_{i_1}...x_{i_m}$. Suppose w.l.g., that $f\sredr f_1$ is an
application of reduction to $x_{i_j}x_{i_{j+1}}=x_1x_0$ and  $f\sredr f_2$
$x_{i_k}x_{i_{k+1}}=x_1x_0$ for $k>j+1$. Then
$f_1=ax_{i_1}\dots x_{i_{j-1}}qx_{i_{j+2}}\dots x_{i_m}$ and $f_2=ax_{i_1}\dots x_{i_{k-1}}qx_{i_{k+2}}\dots x_{i_m}$, where
$q= x_0x_{1}-Q_0(x_{1})-pQ_1(x_0,x_{1})$. 
One may check that $f_1\redr f_3$ and $f_2\redr f_3$ for  
$f_3=ax_{i_1}\dots x_{i_{j-1}}qx_{i_{j+2}}\dots x_{i_{k-1}}qx_{i_{k+2}}\dots x_{i_m}$.   
\section{Proof of Lemma~\ref{lm_free}}\label{sec_lm_free}

Let $\Z(\bar x)$ be an algebra of power series with noncomutative (but associative) 
variables $\bar x=x_0,x_1,\dots,x_m$ over $\Z$. For $a,b\in\Z(\bar x)$ let 
$\lceil a,b\rceil =ab-ba$ and $\Lambda[\bar x]$ be a submodule of $\Z[\bar x]$ generated by
$\lceil \cdot,\cdot\rceil$ starting from $\bar x$. 
Let $\Lambda^j \subset\Lambda[\bar x]$ consist of uniform polynomials of order $j$. So,
$$
\Lambda[\bar x]=\cup_{j=0}^\infty \Lambda^j.   
$$
Let $I=I(\bar x)$ be a (two-sided) ideal in $\Z(\bar x)$ generated by $\bar x$. Clearly,
this ideal consists of polynomials without constant term. 
Let $G$ be a group. Notations $G_n$ and $G_{[n]}$ are defined in 
Section~\ref{sec_p-quotient}.  
\begin{theorem}[Magnus' theorem]\label{th_Magnus}
\begin{itemize}
\item Let $a_i=1+x_i$. The group $F=\langle a_i\rangle$ is a free group, freely
generated by $a_i$.  
\item $(1+I^n)\cap F=F_{[n]}$.
\item If $w \in F_{[n]}$ then $w=1+d+z$, where $d\in\Lambda^n$ and $z$ does not
contain terms of order $\leq$ $n$.
\item For any $d\in\Lambda^n$ there exists $z\in\Z[\bar x]$ without terms of order 
$\leq$ $n$ such that $1+d+z\in F_{[n]}$.
\end{itemize}
\end{theorem}
Consider homomorphism $\pi:\Z(\bar x)\to \Z(\bar x)$, defined by $\pi(x_i)=px_i$. Clearly,
$\pi(\Z(\bar x))=\Z(p\bar x)$. Also, $\pi(F)$ is an inclusion of a free
group $F$ into $\Z(p\bar x)$. For a two sided  ideal $J$ of $\Z(p\bar x)$ let
$N_j=\{w\in\pi(F)\;|\;w-1\in J\}$. 
\begin{lemma}\label{lm_ideal_normal}
$N_j\triangleleft\pi(F)$.
\end{lemma} 
Clearly, $\Z(\bar x)/p^n\Z(\bar x)\equiv \Z_{p^n}(\bar x)$. Denote 
$(p^n)=p^n\Z(\bar x)\cap \Z(\bar px)$. Notice, that $\Z_{p^n}(p\bar x)=\Z_{p^n}[p\bar x]$. 
Now, Lemma~\ref{lm_free} is a consequence of
the following theorem.
\begin{theorem}[Jacobson, \cite{Jacobson}]\label{th_Jacobson}
$N_{(p^n)}=\pi(F_n)$
\end{theorem} 
\begin{proof}
We present here the Jacobson proof (see \cite{Jacobson}) which is a reduction to 
Theorem~\ref{th_Magnus}. In \cite{Jacobson} the definition of $(p^n)$ is different and
not equivalent of ours. But the proof in \cite{Jacobson} is, actually, 
for our definition of $(p^n)$.

Notice, that $(1+p^m w\dots)^p=1+p^{m+1}\dots$ and 
$[(1+p^mw\dots),(1+p^ku\dots)]=1+p^{m+k}(wu-uw)\dots$ 
where the omitted terms are of higher $p$-order. This implies that
$N_{(p^m)}^p\subseteq N_{(p^{m+1})}$ and $[N_{(p^m)},N_{(p^k)}]\subseteq N_{(p^{m+1})}$.
Consequently, we have $\pi(F_n)\subseteq N_{(p^n)}$ and $N_{(p^{n+1})}\subseteq N_{(p^n)}$. 

According to 
\cite{Jacobson} we show the equality $N_{(p^n)}=\pi(F_n)$ by induction.
By definition, $N_{(p^1)}=\pi(F_1)=\pi(F)$. Suppose, that $N_{(p^n)}=\pi(F_n)$.  
Then we know that $\pi(F_{n+1})\subseteq N_{(p^{n+1})}\subseteq \pi(F_n)$. So, it suffices
to show that $\pi(F_n)\cap\pi(F_{n+1})\supseteq N_{(p^{n+1})}\cap\pi(F_n)$, or, the same,
to prove that if $w\in\pi(F_n)\setminus\pi(F_{n+1})$ then $w\not\in N_{(p^{n+1})}$.
Let $w\in\pi(F_n)\setminus\pi(F_{n+1})$. There exists a unique $i$ such that 
$w\in \pi(F_{[i]})\setminus\pi(F_{[i+1]})$. By Theorem~\ref{th_Magnus} 
$w=1+p^jd_i(p\bar x)+z$ where $d_i\in \Lambda^i$, $z$ has $\bar x$-order more than $i$. 
Also we have 
that $i+j\geq n$ (as $w\in N_{(p^n)}$). Applying once again Theorem~\ref{th_Magnus}
we find $u=1+d_i(p\bar x)+z'\in\pi(F_{[i]})$, where the $\bar x$-order of $z'$ is more 
than $i$. So, $w=u^{p^j}w_1$, where $w_1\in \pi(F_{[i']})$ for $i'>i$. 
Repeating this procedure
one gets $w=u_1^{p^{j_1}}u_2^{p^{j_2}}\dots u_k^{p^{j_k}}w_k$, where  $w_k\in \pi(F_{[n+1]})$ and
$u_r=(1+d_{i_r}(p\bar x)\dots)$ with $d_{i_r}\in\Lambda^{i_r}$ (term of higher order in
$\bar x$ are omitted) and $i_r+j_r\geq n$. Now, $u_r\in \pi(F_{[i_r]})$ and, consequently, 
$u_r^{p^{j_r}}\in \pi(F_{[i_r]}^{p^{j_r}}\subseteq F_{i_r+j_r})$. If $\forall r\;\;i_r+j_r>n$
then $w\in \pi(F_{n+1})$, so, by our assumptions, $i_r+j_r=n$ for some $r$. It implies
that $w=u_1^{p^{j_1}}\dots u_k^{p^{j_k}}w_k=
(1+p^{j_1}d_{i_1}(p\bar x)+p^{j_2}d_{i_2}(p\bar x)\dots)\not\in N_{(p^{n+1})}$.  
\end{proof}
Let $F=\langle a_0,\dots,a_m\rangle$ be a free group on $\{a_0,\dots,a_m\}$. 
Let $\Z_{p^n}[F/F_n]$ be a group algebra of $F/F_n$ over $\Z_{p^n}$.
Theorem~\ref{th_Jacobson}
implies that there exists unique homomorphism $\phi:\Z_{p^n}[F/F_n]\to \Z_{p^n}[p\bar x]$
such that $\phi(a_i)=1+px_i$. 
(Here we, abusing notation, denote by the same symbol $a_i$ its image in $F/F_n$.)
Moreover, $(\ker(\phi)+1)\cap F/F_n=\{1\}$. Still $\ker(\phi)$ is not trivial.
For example, if $w\in F/F_n$ and $w^{p^j}=1$ then $p^j(w-1)\in\ker(\phi)$. 
What is the structure of $\ker(\phi)$? For example, is it true that 
$\ker(\phi)$ is generated by $\{p^j(w-1)\;|\;w\in F/F_n,\;w^{p^j}=1\}$?

\end{document}